\documentclass[]{article}
\usepackage{anyfontsize}
\usepackage{graphicx, psfrag, amssymb, amsmath, amsthm, bbm, amsthm}
\usepackage{hyperref}
\usepackage{comment}
\usepackage{xcolor}
\usepackage{makeidx}
\usepackage{pdfsync}
\usepackage{parskip}

\newtheorem{thm}{Theorem}[section]

\newtheorem{lemma}[thm]{Lemma}

\newtheorem{proposition}[thm]{Proposition}
\newtheorem{remark}[thm]{Remark}

\newtheorem*{thma}{Theorem A}
\newtheorem*{thmb}{Theorem B}  

\newtheorem*{corc}{Corollary C}
\newtheorem*{thmd}{Theorem D}

\theoremstyle{definition}
\newtheorem{definition}{Definition}

\numberwithin{equation}{section}

\def\beq{\begin{equation}}
\def\eeq{\end{equation}}
\def\beqn{\begin{equation*}}
\def\eeqn{\end{equation*}}

\def\bB{\mathbf{B}}

\def\bD{\mathbf{D}}
\def\bE{\mathbf{E}}
\def\bF{\mathbf{F}}
\def\bG{\mathbf{G}}

\def\bL{\mathbf{L}}

\def\bQ{\mathbf{Q}}
\def\bR{\mathbf{R}}

\def\ba{\mathbf{a}}

\def\bc{\mathbf{c}}

\def\bk{\mathbf{k}}

\def\bn{\mathbf{n}}

\def\bp{\mathbf{p}}
\def\bq{\mathbf{q}}
\def\br{\mathbf{r}}

\def\bt{\mathbf{t}}

\def\bx{\mathbf{x}}

\def\bLambda{{\boldsymbol{\Lambda}}}

\def\btau{{\boldsymbol{\tau}}}

\def\bdelta{{\boldsymbol{\delta}}}

\def\b1{{\boldsymbol{1}}}

\def\cA{\mathcal{A}}

\def\cC{\mathcal{C}}

\def\cE{\mathcal{E}}
\def\cF{\mathcal{F}}
\def\cG{\mathcal{G}}
\def\cH{\mathcal{H}}

\def\cK{\mathcal{K}}
\def\cL{\mathcal{L}}

\def\cP{\mathcal{P}}

\def\cR{\mathcal{R}}
\def\cS{\mathcal{S}}
\def\cT{\mathcal{T}}

\def\IH{{\mathbb H}}

\def\IN{{\mathbb N}}

\def\IR{{\mathbb R}}

\def\IZ{{\mathbb Z}}

\def\hW{\hat{W}}
\def\hY{\hat{Y}}

\def\eps{\varepsilon}

\def\eps{\varepsilon}

\def\ds{\displaystyle}

\def\ab{\underline{a}}
\def\cb{\underline{c}}

\def\ovarphi{\overline{\varphi}}

\def\bt{\overline{t}}
\def\tb{\underline{t}}

\def\hW{\widehat{W}}

\def\bLambda{{\boldsymbol{\Lambda}}}
\def\bGamma{{\boldsymbol{\Gamma}}}
\def\bOmega{{\boldsymbol{\Omega}}}
\def\btau{{\boldsymbol{\tau}}}
\def\bcK{{\boldsymbol{\cK}}}
\def\bcE{{\boldsymbol{\cE}}}
\def\bvf{{\boldsymbol{\varphi}}}
\def\beps{{\boldsymbol{\eps}}}
\def\bdelta{{\boldsymbol{\delta}}}

\def\ua{\underline{a}}
\def\ub{\underline{b}}
\def\ux{\underline{x}}
\def\opsi{\overline{\psi}}

\def\CM{\mathcal{M}}
\def\wcP{\widetilde{\cP}}
\def\hY{\widehat{Y}}
\def\hf{\widehat{f}}
\def\hE{\widehat{E}}

\def\ba{\mathbf{a}}

\def\os{\overline{s}}
\def\us{\underline{s}}

\begin{document}

\title{Inducing Schemes with Finite Weighted Complexity}

\author{Jianyu Chen\thanks{School of Mathematical Sciences
\& Center for Dynamical Systems and Differential Equations,
Soochow University, Suzhou, Jiangsu 215006, P.R.China. Email: jychen@suda.edu.cn.}
\and Fang Wang \thanks{School of Mathematical Sciences, Capital Normal University, Beijing, 100048, P.R.China.  
Email: fangwang@cnu.edu.cn.}
\and Hong-Kun Zhang\thanks{Department of Mathematics and Statistics, University of Massachusetts,  Amherst MA 01003, USA. Email: hongkun@math.umass.edu. }}

\date{\today}

\maketitle

\begin{abstract}
In this paper, we consider a Borel measurable map of a compact metric space
which admits an inducing scheme. 
Under the finite weighted complexity condition, 
we establish a thermodynamic formalism for a parameter family of potentials 
$\varphi+t\psi$ in an interval containing $t=0$. 
Furthermore, if there is a generating partition compatible to the inducing scheme,
we show that all ergodic invariant measures with sufficiently large pressure are liftable. 

\medskip

\noindent \textbf{Keywords:}
Inducing schemes, 
Weighted Complexity,
Equilibrium measures.
\end{abstract}

\maketitle

\tableofcontents

\section{Introduction}

The main goal in the thermodynamic formalism is to study 
the \emph{equilibrium measures} of a dynamical system $f: X\to X$
for a \emph{potential function} $\varphi: X\to \IR$, i.e., 
the measures for which the supremum 
\beqn
\sup_{\mu\in \CM(f, X)} P_\mu(\varphi)
\eeqn
is attained, where  $\CM(f, X)$ is the class of all $f$-invariant ergodic Borel probability measure on $X$, and
$P_\mu(\varphi)$ is 
the \emph{free energy} given by 
\beq\label{def Pmu}
P_\mu(\varphi):= h_\mu(f) + \int \varphi d\mu.
\eeq
Here $h_\mu(f)$ denotes the Kolmogorov-Sinai entropy of the system $(f, \mu)$. 

The classical works of Sinai, Ruelle and Bowen (see e.g.~\cite{Sin72, Ru78, Bow08}) demonstrated
the existence, uniqueness and ergodic properties of equilibrium measures for 
uniformly hyperbolic systems.
Great efforts have later been made for systems 
beyond uniform hyperbolicity,
using various \emph{extension/inducing} techniques 
(see 
e.g.~\cite{Hof79, BS80, Hof81, Ke89, BSC90, Bruin95, Y98, Y99, Buzzi99, CZ05a, CZ05b, BrTo07, De10, DZ13, DZ14}).
Using principle results obtained by Sarig (see e.g.~\cite{Sarig99, Sarig01, Sarig03})
on the thermodynamic formalism for the countable Markov shifts,
Pesin, Senti and Zhang \cite{PS05, PS08, PSZ08, PSZ16} developed
a version of the \emph{inducing scheme} method,
which is applicable to some multimodal interval maps, the
Young's diffeomorphisms, the H\'enon family and the Katok map.

In this paper, we consider a Borel measurable map $f: X\to X$
of a compact metric space, possibly with discontinuities and singularities. 
We assume that the map $f$  
admits an \emph{inducing scheme} $\{\cS, \tau\}$
satisfying Conditions \textbf{(I1)-(I4)}  (see Definition~\ref{def is}),
which is either of hyperbolic type or of expanding type.
The inducing scheme $\{\cS, \tau\}$
might not have \emph{finite complexity} (see~\eqref{def finite comp}),
in particular, there may be infinitely many blocks with the same inducing time. 
In such situation, we could only expect  
some \emph{finite weighted complexity} condition (see Definition~\ref{def finite psi comp}) 
for  particular weight functions.

In the context of continuous maps admitting inducing schemes, 
Pesin, Senti and Zhang \cite{PS05, PS08, PSZ08, PSZ16} have established a thermodynamic formalism 
with respect to a class of \emph{nice} potential functions,
i.e., functions satisfying the verifiable conditions \textbf{(P1)-(P4)} in Section 4 of \cite{PSZ16} 
(see Definition~\ref{def regular nice} for the precise statements). 
In this paper, we adapt their results to a measurable map $f: X\to X$ with discontinuities and singularities,
which is assumed to have finite \emph{variational entropy} (see \eqref{def entropy}). 
Once an inducing scheme $\{\cS, \tau\}$ is chosen, 
we denote $Y$ the forward invariant subset associated with $\{\cS, \tau\}$  (see \eqref{def Y}). 
The potential functions under consideration is a one parameter family
$\{\varphi+t\psi\}_{t\in \IR}$, where $\varphi$ is a \emph{strongly nice} potential 
and $\psi$ is a \emph{regular} potential 
(see Definition~\ref{def regular nice}).
Under Condition \textbf{(H)} below, which assumes the
finite weighted complexity (see Definition~\ref{def finite psi comp}) at two parameters $\tb_0<0<\bt_0$, 
we are able to establish a thermodynamic formalism 
for the potentials $\varphi+t\psi$   in an interval containing $t=0$.
Applying a result by Shahidi and  Zelerowicz \cite{ShZe19},
we also obtain the analyticity of the \emph{variational lifted pressure} (see \eqref{def lift pressure})
in this interval.

The precise statements of our first main result is as follows.

\begin{thma} 
Let $\{\cS, \tau\}$ be an inducing scheme.
Assume that  $\varphi$ is a strongly nice potential  
and  $\psi$ is a regular potential.
Furthermore,  we assume that
\begin{enumerate}
\item[\textbf{(H)}] 
there are $\tb_0<0<\bt_0$ such that the inducing scheme  $\{\cS, \tau\}$
has finite $\overline{\varphi+t\psi}$-complexity (see Definition~\ref{def finite psi comp}) 
at $t=\tb_0$ and $t=\bt_0$, 
where $\overline{\varphi+t\psi}$ is the induced potential of $\varphi+t\psi$
(see~\eqref{def induced potential}).
\end{enumerate}
Then there exist $\tb<0<\bt$ such that the following hold:
\begin{enumerate}
\item[(1)] 
for every $t\in (\tb, \bt)$,
there exists a unique equilibrium measure $\mu_{t}$
for the potential $\varphi+t\psi$ 
in the class $\CM_L(f, Y)$ of all liftable measures (see \eqref{def lift class});
\item[(2)] 
if the inducing scheme $\{\cS, \tau\}$ satisfies the aperiodic condition~\eqref{def aperiodic}, 
then the  measure $\mu_{t}$ is mixing and in fact Bernoulli. Moreover, $\mu_t$
has exponential decay of correlations and satisfies the 
Central Limit Theorem (with respect to a class of functions 
which contains all bounded H\"older continuous functions on $Y$).
\item[(3)] 
the  function $\ds t\mapsto P_L(\varphi+t\psi)$
is finite and real analytic in $(\tb, \bt)$, where $P_L(\cdot)$ is the 
variational liftable pressure  given by \eqref{def lift pressure}.
\end{enumerate}
\end{thma}

We remark that $\tb$ and $\bt$ are chosen by the explicit formula in~\eqref{choose tb bt}.

\medskip

Our second result deals with the \emph{liftability problem} over the inducing schemes,
that is, we would like to determine whether an ergodic invariant measure $\mu\in \CM(f, X)$
is \emph{liftable}, i.e., $\mu\in \CM_L(f, Y)$. 
For continuous maps admitting inducing schemes, 
Pesin, Senti and Zhang \cite{PSZ08, PSZ16} 
provide some conditions under which all ergodic invariant measures of sufficiently large entropy are liftable. 
A crucial assumption therein is the finite complexity for non-first-return inducing times  
(see Condition (L2) in Theorem 5.1 of \cite{PSZ16}). 
For systems with nasty discontinuities and singularities, one could only impose 
certain finite weighted complexity condition like Condition \textbf{(L2)} below. 
Together with Condition \textbf{(L1)} below, which asserts the existence of a finite generating partition
\emph{compatible} to the inducing scheme  (see Definition~\ref{def compatible partition}),
we are able to show that all ergodic invariant measures with 
sufficiently large free energy are liftable.

Our second main result is stated as follows.

\begin{thmb}
Let $\{\cS, \tau\}$ be an inducing scheme.
We assume that
\begin{enumerate}
\item[\textbf{(L1)}]
the map $f: X\to X$ has a generating partition $\cP$, which is compatible to 
the inducing scheme $\{\cS, \tau\}$ (see Definition~\ref{def compatible partition});
\item[\textbf{(L2)}] 
the inducing scheme $\{\cS, \tau\}$ has finite $\ovarphi$-complexity  (see Definition~\ref{def finite psi comp}),
for some potential function $\varphi: X\to \IR$.
\end{enumerate}
Then any ergodic measure $\mu\in \CM(f, X)$ 
with $\mu(W)>0$ and 
$\cK(\ovarphi)<P_\mu(\varphi)<\infty$ is liftable,
where  $\cK(\ovarphi)$ and $P_\mu(\varphi)$ are given by  
Definition~\ref{def finite psi comp} and \eqref{def Pmu} respectively.
\end{thmb}

\medskip

An  immediate corollary of  
Theorem~A and Theorem~B is the following. 
 
\begin{corc}
Let $\{\cS, \tau\}$ be an inducing scheme with inducing domain $W$.
Assume that  $\varphi$ is a strongly nice potential  
and  $\psi$ is a regular potential.
If Conditions \textbf{(H)} and  \textbf{(L1)} hold, then
the measure $\mu_t$ obtained in Theorem~A is in fact 
the unique equilibrium measure for the potential $\varphi+t\psi$ 
in the class 
\beq\label{def CM*}
\CM^W(f, X):=\left\{\mu\in \CM(f, X):    
  \mu(W)>0 \right\}.
\eeq
Therefore,
the function
$
\ds t\mapsto P^W(\varphi+t\psi) 
$
is finite and real analytic in $(\tb, \bt)$, where
$P^W(\cdot)$ is  the \emph{variational pressure}  given by \eqref{def v pressure}.
\end{corc}

\section{Inducing Schemes}

To work on the thermodynamics for continuous maps of compact metric spaces,
Pesin, Senti and Zhang developed an inducing scheme method in 
\cite{PS05, PS08, PSZ08, PSZ16}. 
Our goal is to adapt this general framework to systems with discontinuities and singularities,
for which the finite complexity condition (see~\eqref{def finite comp}) usually fails. To this end,
we first introduce the basic notations as follows.

Let $f: X\to X$ be a measurable map of a compact metric space,
which is allowed to have discontinuities. 
If further $X$ is a  smooth manifold,
there may be singularity points at which $f$ is continuous but $Df$ does not exist. 
We assume that $f$ has finite \emph{variational entropy}, that is,
\beq\label{def entropy}
h(f):=\sup_{\mu\in \CM(f, X)} h_\mu(f)<\infty,
\eeq
where $\CM(f, X)$ is the set of all $f$-invariant ergodic Borel probability measures on $X$
and $h_\mu(f)$ is the Kolmogorov-Sinai entropy of the system $(f, \mu)$.

The definition of inducing schemes is given as follows.

\begin{definition} 
\label{def is}

Given a countable collection of disjoint Borel subsets $\cS=\{J_a\}_{a\in S}$ and 
a positive integer value function $\tau: S\to \IN$, 
we say that $f$ admits an \emph{inducing scheme  $\{\cS, \tau\}$ of hyperbolic type}, with 
the \emph{inducing domain} $W:=\bigcup_{a\in S} J_a$ and 
the \emph{inducing time} $\tau: X\to \IN$ defined by 
$\ds
 \tau(x):=
 \begin{cases}
 \tau(a),  & x\in J_a, \\
 0,  & x\not\in W,
 \end{cases}
$
provided the following conditions \textbf{(I1)-(I4)} hold: 
\begin{enumerate}
\item[\textbf{(I1)}] 
For any $a\in S$, we have 
\beq\label{is I1}
f^{\tau(a)}(J_a)\subset W \ \ \text{and} \ \ \bigcup_{a\in S} f^{\tau(a)}(J_a)=W. 
\eeq
Moreover, let $F: W\to W$ be the \emph{induced map} defined by
\beqn
F|J_a= f^{\tau(a)}|J_a \ \ \text{for any}\ a\in S, 
\eeqn
and assume that 
$F|J_a$ can be extended to a homeomorphism of a neighborhood of 
the closure $\overline{J_a}$.
\item[\textbf{(I2)}] 
For every bi-infinite sequence $\ua=(a_n)_{n\in\IZ} \in S^\IZ$, there exists a unique
sequence $\ux=(x_n)_{n\in \IZ}$ with coordinates $x_n=x_n(\ua)$ such that
\begin{itemize}
\item[(a)] $\ds x_n\in \overline{J_{a_n}}$ and 
$\ds f^{\tau(J_{a_n})}(x_n)=x_{n+1}$;
\item[(b)]
if $\ds x_n(\ua)=x_n(\ub)$ for all $n\le 0$, then $\ua=\ub$. 
\end{itemize}
\end{enumerate} 
Denote by $\sigma: S^\IZ\to S^\IZ$ the full left shift and let
\beqn
\check{S}:=\left\{ \ua\in S^\IZ: \ x_n(\ua)\in J_{a_n}\ \text{for all}\ n\in \IZ \right\}.
\eeqn
\begin{enumerate}
\item[\textbf{(I3)}] 
The set $S^\IZ\backslash \check{S}$ supports no ergodic $\sigma$-invariant 
measure which gives positive mass to any open subset.
\item[\textbf{(I4)}]
The induced map $F$ has at least one periodic orbit in $W$.  
\end{enumerate} 
\end{definition}

\medskip

\begin{remark}
We say $f$ admits an \emph{inducing scheme  $\{\cS, \tau\}$ of expanding type}
if Conditions \textbf{(I1)-(I4)} hold, with the following changes:
\begin{itemize}
\item[(1)]
the second equation of \eqref{is I1} in \textbf{(I1)} is replaced by $f^{\tau(a)}(J_a)=W$;
\item[(2)]
the two-sided infinite sequence in $S^\IZ$ is replaced by one-sided  infinite sequence in $S^{\IN_0}$,
and Condition \textbf{(I2)}(b) is removed. 
\end{itemize}
\end{remark}

In this paper, we shall only state and prove results for inducing schemes of hyperbolic type, 
as all the results for the expanding type can be obtained in a similar fashion,
thereby we omit the term ``of hyperbolic type'' when we mention inducing schemes.

We say that an inducing scheme $\{\cS, \tau\}$ satisfies the \emph{aperiodicity} condition if 
\beq\label{def aperiodic}
\ds gcd\left\{\tau(a):  a\in S \right\}=1.
\eeq

In what follows, we shall fix an inducing scheme $\{\cS, \tau\}$ and 
explain the definitions of weighted complexity, 
liftable measures and regular/nice potentials, 
and then we state a main theorem in \cite{PS08, PSZ16} concerning
on the thermodynamics of the inducing scheme $\{\cS, \tau\}$.

\subsection{Weighted Complexity}

For any $n\in \IN$, we set 
\beq\label{def Sn}
S_n:=\{a\in S: \ \tau(a)=n\}.
\eeq
We say that the inducing scheme $\{\cS, \tau \}$ has \emph{finite complexity} if 
\beq\label{def finite comp}
\# S_n<\infty \ \text{for all} \ n\ge 1, 
\ \text{and} \ 
\limsup_{n\to\infty} \frac1n \log \# S_n <\infty.
\eeq
Due to the presence of discontinuities and singularities, 
the finite complexity condition~\eqref{def finite comp} usually fails.
In fact, when $X$ is a smooth manifold and $Df$ blows up near the singularity points,
it is very likely that $\# S_n=\infty$ for infinitely many $n$'s,
or every $\# S_n$ is finite but grows super-exponentially fast.

Nevertheless, when we deal with some specific class of potentials, 
it would be useful to introduce the following concept.

\begin{definition}\label{def finite psi comp}
Given a weight function $\omega:  W\to \IR$,  
the inducing scheme $\{\cS, \tau \}$ is said to have
\emph{finite $\omega$-complexity} if 
\begin{itemize}
\item[(a)] $\ds U_n(\omega):=\sum_{a\in S_n} \sup_{x\in  J_a} \exp{(\omega(x))}<\infty$ for all $n\ge 1$;
\item[(b)] $\ds \cK(\omega):=\limsup_{n\to\infty} \frac1n \log  U_n(\omega) <\infty.$
\end{itemize}
 We shall call $\cK(\omega)$ the \emph{complexity function} of $\omega$. 
\end{definition}
 
\medskip

It is clear that the finite complexity condition~\eqref{def finite comp} 
is a special case of the finite weighted complexity condition
with the weight function $\omega\equiv 0$.

\subsection{Pressure and Equilibrium Measures}

The \emph{forward invariant hull} of the inducing domain $W$ is defined by 
\beq\label{def Y}
Y:=\bigcup_{k=0}^\infty f^k(W)=
\left\{ f^n(x):\ x\in W, \ 0\le n< \tau(x) \right\}.
\eeq
Denote the set of $f$-invariant ergodic Borel probability measures on $X$ (or $Y$) by 
$\CM(f, X)$ (or $\CM(f, Y)$),
and denote the set of $F$-invariant ergodic Borel probability { measures} on $W$ by $\CM(F, W)$.
For any $\nu\in \CM(F, W)$, let
\beq\label{def Qnu}
Q_\nu:=\int_{W} \tau\ d\nu.
\eeq
If $Q_\nu<\infty$, then $\nu$ is \emph{liftable} to a measure $\mu=\cL(\nu)\in \CM(f, Y)$ given by
\begin{equation*}
\mu(E):=\frac{1}{Q_\nu} \sum_{a\in S} \sum^{\tau(a)-1}_{k=0} \nu(f^{-k}E\cap J_a)
\end{equation*}
for any Borel subset $E\subset Y$.
We denote  the class of liftable measures  by
\beq\label{def lift class}
\CM_L(f, Y):=\left\{\mu\in \CM(f, Y): \ 
\text{there is}\ \nu\in \CM(F, W) \ \text{with}\ \cL(\nu)=\mu
\right\}.
\eeq
The 
\emph{variational liftable pressure} of a function $\varphi: X\to \IR$ is defined by
\beq\label{def lift pressure}
P_L(\varphi):=\sup_{\mu\in \CM_L(f, Y)} P_\mu(\varphi),
\eeq
where $P_\mu(\varphi)$ is the \emph{pressure of $\varphi$ with respect to $\mu$} given by~\eqref{def Pmu}. 
Condition \textbf{(I4)} implies that  $\CM_L(f, Y)\ne \emptyset$  
and $P_L(\varphi)>-\infty$.  Note that it is possible that 
$P_L(\varphi)=+\infty$.
Also, $P_L(\varphi-q)=P_L(\varphi)-q$ for any $q\in \IR$.

A measure 
$\mu\in \CM_L(f, Y)$ is called an \emph{equilibrium measure} of $\varphi$ 
(in the class of $\CM_L(f, Y)$ of liftable measures)
if it attains the supremum of \eqref{def lift pressure}. 
Note that the potentials $\varphi-q$ and $\varphi$ share the 
same equilibrium measures for any $q\in \IR$.

We remark that the  $f$-invariant hull $Y$, 
the measure class $\CM_L(f, Y)$ and the variational liftable pressure $P_L(\varphi)$
depend on the choice of the inducing schemes $\{\cS, \tau\}$. 
A more independent quantity is the \emph{variational pressure} given by
\beq\label{def v pressure}
P^W(\varphi):=\sup\left\{  P_\mu(\varphi):  \ \mu\in \CM^W(f, X) 
\ \text{with} \ \int |\varphi| d\mu<\infty \right\},
\eeq
where the class $\CM^W(f, X)$ of measures is given by~\eqref{def CM*}.
Notice that the only requirement on $\CM^W(f, X)$
is to give positive mass on the inducing domain $W$.

\subsection{Regular and Nice Potentials} 

Condition \textbf{(I2)} allows one to define the \emph{coding map}
$\pi: S^\IZ\to \hW:=\bigcup_{a\in S} \overline{J_a}$ by
\beqn 
\pi(\ua):=x_0(\ua)=\bigcap_{n=-\infty}^\infty 
F^{-n}\overline{J_{a_n}}
\eeqn
for every bi-infinite sequence $\ua=(a_n)_{n\in\IZ} \in S^\IZ$. 
Note that $\pi$ is a semi-conjugacy from $(S^\IZ, \sigma)$ to $(\hW, F)$, 
i.e., $\pi\circ \sigma=F\circ \pi$,
and it is one-to-one on $\check{S}$.

Let $\varphi: X\to \IR$ be a potential function. 
In below, we list some verifiable conditions 
that were introduced in \cite{PS08, PSZ16}.

\begin{enumerate}
\item[\textbf{(P1)}] 
the \emph{induced potential} $\ovarphi: W\to \IR$ given by
\beq\label{def induced potential}
\ovarphi(x):=\sum^{\tau(x)-1}_{k=0} \varphi(f^k(x))
\eeq
can be extended
by continuity to a function on $\overline{J_a}$ for any $a\in S$;
\item[\textbf{(P2)}]
the function $\Phi=\ovarphi\circ \pi$ is a \emph{locally H\"older continuous} function 
on $S^\IZ$, i.e., there are $H>0$ and $0<r<1$ such that for all $n\ge 1$, 
\beqn
V_n(\Phi):=
\sup_{[\cb]\in \cC_{-n+1}^{n-1}}\ \sup_{\ab, \ab'\in [\cb]} \left\{\left|\Phi(\ab)-\Phi(\ab')\right| \right\}
\le H r^n,
\eeqn
where
$\cC_n^m$, with integers $n\le m$, is the collection of cylinders of the form 
\beqn
[\cb]=[c_n \dots  c_{m}]:=
\left\{\ua\in S^\IZ: \ a_k=c_k, \ k=n, \dots, m\right\}.
\eeqn 
 
\item[\textbf{(P3)}] 
the following summability condition holds:
\beqn
\ds \sum_{a\in S} \sup_{x\in  J_a} \exp \ovarphi(x)<\infty.
\eeqn
\end{enumerate}

By  Theorem 4.5 in \cite{PS08}, if $\varphi$ satisfies \textbf{(P1)-(P3)},
then $P_L(\varphi)<\infty$. 
We then define the \emph{normalized induced potential} $\varphi^+: W\to \IR$
by
\beqn
\varphi^+:=\overline{\varphi-P_L(\varphi)}=\ovarphi-P_L(\varphi)\tau.
\eeqn
\begin{enumerate}
\item[\textbf{(P4)}]
there exists $\eps>0$ such that 
\beqn
\ds\sum_{a\in S} \tau(a) \sup_{x\in  J_a} \exp (\varphi^+(x)+\eps \tau(x)) <\infty.
\eeqn
\end{enumerate}

Again by Theorem 4.5 in \cite{PS08}, if $\varphi$ satisfies \textbf{(P1)-(P4)},
the normalized induced potential $\varphi^+$ has zero topological pressure, and
there is a unique $F$-invariant equilibrium measure $\nu_{\varphi^+}$ for $\varphi^+$
among all the measures in $\CM(F, W)$. Moreover, 
$\nu_{\varphi^+}$ has the \emph{Gibbs property}, i.e.,
there exists a constant $K>0$ such that for any $n\ge 1$,
any cylinder $[\cb]\in \cC_0^{n-1}$ and any $x\in J_{\cb}:=\pi[\cb]$, 
we have 
\beqn
K^{-1}\le \dfrac{\nu_{\varphi^+}(J_{\cb})}{\exp \left(\sum_{k=0}^{n-1}\varphi^+(F^k(x))\right)} \le K.
\eeqn
In particular, for any $a\in S$ and any $x\in J_a$, we have 
\beq\label{Gibbs1}
K^{-1}\le\dfrac{\nu_{\varphi^+}(J_a)}{\exp \left( \varphi^+(x)\right)} \le K.
\eeq

\noindent\textbf{(P5)}
We say that $\nu_{\varphi^+}$
 has \emph{exponential tail}, if  
 there are $C>0$ and $\theta\in (0, 1)$ such that 
 for any $n\ge 1$, 
\beqn\label{nu exp tail}
\nu_{\varphi^+}(\{x\in W: \ \tau(x)\ge n\})\le C\theta^n.
\eeqn

\medskip

We now introduce the following definitions for potential functions.

\begin{definition}
\label{def regular nice}
We say that a function $\varphi: X\to \IR$ is  
\begin{itemize}
\item[(1)] \emph{regular} if $\varphi$ satisfies Conditions \textbf{(P1)(P2)};
\item[(2)] \emph{nice} if $\varphi$ satisfies Conditions \textbf{(P1)-(P4)}; 
\item[(3)] \emph{strongly nice} if $\varphi$ satisfies Conditions \textbf{(P1)-(P5)}.
\end{itemize}
\end{definition}

\medskip

We remark that the above properties depend on the choice of the inducing scheme.
It is possible that a function is regular with respect to one inducing scheme but not regular
with respect to some others. 
We also stress that the induced potential 
$\ovarphi$ may be unbounded even if $\varphi$ is a regular potential.

\subsection{Thermodynamics of Inducing Schemes}

Let us describe some statistical properties of $\mu\in \CM(f, X)$. 
We say that $\mu$ has \emph{exponential decay of correlations}
with respect to a class $\cH$ of functions if there exists $\Theta\in (0, 1)$
such that for any $h_1, h_2\in \cH$ and any $n\ge 1$,
\beqn
\left| \int h_1 \cdot h_2\circ f^n d\mu - \int h_1 d\mu \int h_2 d\mu \right|\le K\Theta^n,
\eeqn
for some $K=K(h_1, h_2)>0$. 
We say that $\mu$ satisfies the \emph{Central Limit Theorem} with respect to the class $\cH$ 
if for any $h\in \cH$, which is not cohomologous to a constant,
there exists $\sigma> 0$ such that 
\beqn
\frac{1}{\sqrt n} \ds \sum_{k=0}^{n-1} \left( h\circ f^k -  \int h d\mu \right) 
\longrightarrow N(0, \sigma^2) \ \ \text{in law},
\eeqn
where $N(0, \sigma^2)$ denotes the normal distribution.

We now state some main results by Pesin, Senti and Zhang~\cite{PS08, PSZ16}
on the thermodynamic formalism of inducing schemes.
We also include the result on the Bernoulli property established by Shahidi and Zelerowicz \cite{ShZe19}.
Although we are dealing with maps with discontinuities and singularities in this paper,
the proof of the following theorem can still be carried out
by the same arguments.

\begin{thm}\label{thm: is0} 
Let $\{S, \tau\}$ be an inducing scheme, and 
assume that the potential function $\varphi$ is nice.
Then 
\begin{enumerate}
\item[(1)] 
there exists a unique equilibrium measure $\mu_{\varphi}$
for $\varphi$ among all liftable measures in $\CM_L(f, Y)$;
\item[(2)] 
if the inducing scheme $\{\cS, \tau\}$ satisfies the aperiodic condition~\eqref{def aperiodic}, 
then the meausre $\mu_\varphi$ is mixing and in fact Bernoulli.
Furthermore, if
$\nu_{\varphi^+}=\cL^{-1}(\mu_\varphi)$ has exponential tail, 
then $\mu_\varphi$ has exponential decay of correlations and satisfies the 
Central Limit Theorem (with respect to a class of functions 
which contains all bounded  H\"older continuous functions on $Y$).
\end{enumerate}
\end{thm}

\medskip

Recall that $P_L(\cdot)$ is the variational lifted pressure given by \eqref{def lift pressure}.
The following theorem was established by Shahidi and Zelerowicz \cite{ShZe19},
which provides a criterion for the analyticity of the variational lifted pressure.

\begin{thm}\label{thm: SZ}
Let $\varphi_1$ and $\varphi_2$ be two potential functions on $X$.
Assume there exists $\eps_0>0$ such that   $\varphi_1+s\varphi_2$ is nice 
for all $|s|<\eps_0$.
Then for some $0<\eps<\eps_0$, the function $s\mapsto P_L(\phi_1+s\phi_2)$ 
is real analytic in $(-\eps, \eps)$. 
\end{thm}

\section{Thermodyanmic Formalism}
 
\subsection{Auxiliary functions}

Let $\varphi$ be a strongly nice potential satisfying Conditions \textbf{(P1)-(P5)}
and let $\psi$ be a regular potential satisfying Conditions \textbf{(P1)(P2)}.

We need the following auxiliary functions.  
For any $t\in (-\infty, \infty)$, we set
\beqn
\rho_a(t):=\sup_{x\in J_a} \exp \left(\overline{\varphi+t\psi} (x) \right),  
\eeqn
where $\overline{\varphi+t\psi}$ is 
the induced potentials of $\varphi+t\psi$. 
Recall that $S_n$ is the set defined by \eqref{def Sn}
and $U_n(\cdot)$ is given by Definition~\ref{def finite psi comp}.
For any $n\ge 1$, we set
\beq\label{def un}
u_n(t):= U_n\left(\overline{\varphi+t\psi}\right)=
\sum_{a\in S_n}\rho_a(t).
\eeq
We allow $u_n(t)=\infty$ if the above sum diverges,
and we set $u_n(t)=0$ if $S_n$ is empty.  Hence $u_n$ is a function 
from $(-\infty, \infty)$ to $[0, \infty]$.

By convention, we set $\log 0=-\infty$ and $\log\infty=\infty$. 
Recall that  $\cK(\cdot)$ is the complexity function given by Definition~\ref{def finite psi comp}.
We define a function $\kappa: (-\infty, \infty) \to [-\infty, \infty]$ by
\beq\label{define kappa}
\kappa(t):= \cK\left(\overline{\varphi+t\psi}\right)=
\limsup_{n\to \infty} \frac{1}{n}\log u_n(t).
\eeq
  
\begin{lemma}\label{lem: kappa}
The  following statements hold:
\begin{itemize}
\item[(1)] The function $\kappa(t)$ is convex  on $(-\infty, \infty)$;
\item[(2)] If Condition \textbf{(H)} holds, then 
\begin{itemize}
\item[(a)]
$u_n(t)<\infty$ for any $t\in [\tb_0, \bt_0]$ and any $n\ge 1$;
\item[(b)]
either $\kappa\equiv -\infty$ or 
$\kappa$ is a real-valued continuous function on $(\tb_0, \bt_0)$;
\end{itemize}
\item[(3)]  $\kappa(0)<p:=P_L(\varphi)<\infty$.
\end{itemize}
\end{lemma}

\begin{proof} (1) For any $n\in \IN$, we set 
\beq\label{def vn}
v_n(t)=\sum_{a\in S_n} \rho_a \xi_a^t
\eeq
for all $t\in (-\infty, \infty)$, 
where
\beq\label{def rho xi}
\rho_a:=\rho_a(0)=\sup_{x\in J_a} \exp \ovarphi(x) \ \ \text{and} \ \
\xi_a:=\sup_{x\in J_a} \exp \opsi(x).
\eeq
Since the potential $\psi$ satisfies 
Condition \textbf{(P2)}, there is $H>0$ such that for any $a\in S$,  we have
\beq\label{psi osc}
\ds \sup_{x\in J_a} \opsi(x) - \inf_{x\in J_a} \opsi(x) \le H,
\eeq
and hence for any $t\in\IR$,  
\beq\label{relation uv}
v_n(t) e^{-H |t|}\le u_n(t)\le v_n(t) e^{H|t|}.
\eeq
For any $t_1, t_2\in (-\infty, \infty)$ and $\alpha\in (0, 1)$, by the H\"older inequality,
\beqn
v_n(\alpha t_1 + (1-\alpha) t_2)
=\sum_{a\in S_n} \left[\rho_a \xi_a^{ t_1}\right]^{\alpha} \left[ \rho_a \xi_a^{t_2}\right]^{1-\alpha}
\le v_n(t_1)^\alpha v_n(t_2)^{1-\alpha},
\eeqn
and hence 
\beq\label{convex un}
u_n(\alpha t_1 + (1-\alpha) t_2)
\le u_n(t_1)^\alpha u_n(t_2)^{1-\alpha} e^{2H\left[\alpha|t_1| + (1-\alpha)|t_2| \right]}.
\eeq
Taking $\limsup\limits_{n\to\infty} \frac{1}{n}\log$ on both sides of \eqref{convex un}, we have that
\beqn
\kappa(\alpha t_1 + (1-\alpha) t_2)\le \alpha \kappa(t_1) + (1-\alpha) \kappa(t_2).
\eeqn
Hence the function $\kappa$ is convex on $(-\infty, \infty)$.

(2)  It immediately follows from Condition \textbf{(H)} that 
\beqn
\begin{split}
& u_n(\tb_0)<\infty \ \text{for all} \ n\ge 1, \ \ \text{and} \ \ \kappa(\tb_0)<\infty, \\
& u_n(\bt_0)<\infty \ \text{for all} \ n\ge 1, \ \ \text{and} \ \ \kappa(\bt_0)<\infty.
\end{split}
\eeqn
For any $n\ge 1$ and $t\in (\tb_0, \bt_0)$, we can write $t=\alpha \tb_0 + (1-\alpha)\bt_0$ for some $\alpha\in (0, 1)$,
thus $u_n(t)<\infty$ by \eqref{convex un}. 

Furthermore, 
by the convexity of $\kappa$, if $\kappa\not\equiv -\infty$,
then we must have $\kappa(t)>-\infty$ for all $t\in (-\infty, \infty)$.
Since $\kappa(\tb_0)<\infty$ and $\kappa(\bt_0)<\infty$, by convexity of $\kappa$,
we get $\kappa(t)<\infty$ for any $t\in [\tb_0, \bt_0]$.
In other words, 
$\kappa$ is a real-valued convex function on $[\tb_0, \bt_0]$,
and hence $\kappa$ is continuous in the interior $(\tb_0, \bt_0)$.

(3) Recall that $\varphi$ is a strongly nice potential, that is, $\varphi$
satisfies Conditions  \textbf{(P1)-(P5)}.
Conditions \textbf{(P1)-(P3)} imply that $p:=P_L(\varphi)<\infty$.
Together with Condition \textbf{(P4)}, we have that
$\varphi^+=\ovarphi - p \tau$ has a unique $F$-invariant equilibrium measure
$\nu_{\varphi^+}$.
By \eqref{Gibbs1}, there is $K>0$ such that for any $a\in J$,
\beq\label{Gibbs2}
K^{-1}\le \dfrac{\nu_{\varphi^+}(J_a) }{ \rho_a  e^{-p\tau(a)}}\le K,
\eeq
where $\rho_a$ is given by \eqref{def rho xi}. 
Condition  \textbf{(P5)} further says that $\nu_{\varphi^+}$ has exponential tail, i.e.,
there are $C>0$ and $\theta\in (0, 1)$ such that
\beqn
u_n(0)=\sum_{a\in S_n} \rho_a \le Ke^{p n} 
\nu_{\varphi^+}\left(\{x\in W: \tau(x)=n\}\right)
\le CK e^{pn} \theta^n.
\eeqn
Hence
\beqn
\kappa(0)=\limsup\limits_{n\to\infty} \frac{1}{n}\log u_n(0)\le p+\log\theta<p.
\eeqn

The proof of this lemma is complete.
\end{proof}

\subsection{Proof of Theorem~A}

Let $\varphi$ be a strongly nice potential satisfying Conditions \textbf{(P1)-(P5)}
and let $\psi$ be a regular potential satisfying Conditions \textbf{(P1)(P2)}.

We shall consider the potentials of the form
$\varphi+t\psi-q_t$ for all $t\in \IR$, where $q_t$ is a constant with regard to $x$ 
(see the precise formula in \eqref{def qt}).  
Note that the equilibrium measures of $\varphi+t\psi-q_t$
are the same as those of $\varphi+t\psi$ in the liftable class $\CM_L(f, Y)$.

\begin{proposition}\label{lem: check P1-2} For any $t\in \IR$ and any  $q\in \IR$, 
the potential $\varphi+t\psi-q$ satisfies Conditions \textbf{(P1)} and \textbf{(P2)}.
\end{proposition}

\begin{proof}
Note that $\ds \overline{\varphi+t\psi-q}=\ovarphi +t\opsi - q\tau$,
and $\ovarphi$ and $\opsi$ both satisfy Conditions \textbf{(P1)} and \textbf{(P2)}. 
Since $\tau$ is constant on each $J_a$, it is continuous on $\overline{J_a}$
and $\tau\circ \pi$ is automatically locally H\"older continuous. 
Therefore, the potential $\varphi+t\psi-q$ satisfies Conditions \textbf{(P1)} and \textbf{(P2)}.
\end{proof}

In the rest of this section, we show that

\begin{proposition}\label{lem: check P3-5}
Under Condition \textbf{(H)}, 
there exist $\tb<0<\bt$ such that for all $t\in (\tb, \bt)$, there exists $q_t\in \IR$ so that
the potential $\varphi+t\psi-q_t$ satisfies Conditions \textbf{(P3)-(P5)}.
\end{proposition}

From now on, we shall assume that Condition \textbf{(H)} holds.
To prove Proposition \ref{lem: check P3-5}, we need the following preparations.
Set
\beq\label{def lambda}
\lambda:=\int \psi d\mu_\varphi,
\eeq
where  $\mu_\varphi$ is the equilibrium measure for the strongly nice potential $\varphi$,
which is ensured by Theorem~\ref{thm: is0}. 

\begin{lemma} 
The number $\lambda$ given in \eqref{def lambda} is well defined, and
$\ds -\infty<\lambda<\infty$. 
\end{lemma}

\begin{proof}
Since $\mu_\varphi=\cL(\nu_{\varphi^+})\in \CM_L(f, Y)$, we have 
$1\le Q_{\nu_{\varphi^+}}<\infty$, where $Q_{\nu_{\varphi^+}}$ is given by \eqref{def Qnu}.
By Kac's formula (see e.g. Theorem 2.3 in \cite{PS08}), if $\int \opsi d\nu_{\varphi^+}$ is finite, then 
\beqn
-\infty < \int \opsi d\nu_{\varphi^+} = Q_{\nu_{\varphi^+}} \int \psi d\mu_\varphi <\infty.
\eeqn
It then suffices to show that $\ds -\infty < \int \opsi d\nu_{\varphi^+}<\infty$.
By \eqref{psi osc}, 
there is $H>0$ such that  
\beqn
\sum\limits_{a\in S} \sup_{x\in J_a} \opsi(x)\cdot  \nu_{\varphi^+}(J_a) - H \le 
\int \opsi d\nu_{\varphi^+}  
\le \sum\limits_{a\in S} \sup_{x\in J_a} \opsi(x) \cdot\nu_{\varphi^+}(J_a).
\eeqn
Recall the definitions of $\rho_a$ and $\xi_a$ in \eqref{def rho xi},
and also by \eqref{Gibbs2}, 
we have 
\beqn
K^{-1}\le 
\dfrac{\sum\limits_{a\in S} \sup_{x\in J_a} \opsi(x) \cdot \nu_{\varphi^+}(J_a)}{\sum\limits_{a\in S} \rho_a e^{-p\tau(a)} \log \xi_a}
\le K. 
\eeqn
Thus it boils down to show that 
\beqn
-\infty<
\sum\limits_{a\in S} \rho_a e^{-p\tau(a)} \log \xi_a = \sum_{n=1}^\infty e^{-pn} \sum_{a\in S_n} \rho_a\log \xi_a
<\infty.
\eeqn
Recall that $u_n(t)$ and $v_n(t)$ are defined by \eqref{def un} and \eqref{def vn} respectively,
and 
$\kappa(t)$ is  the auxiliary function  defined in \eqref{define kappa}.
By Condition \textbf{(H)} and Lemma~\ref{lem: kappa},  
for any $p_1\in (\kappa(0), p)$, there exist
$\tb_1\in (\tb_0, 0)$ and $\bt_1\in (0, \bt_0)$ 
such that $\kappa(\tb_1)<p_1$ and $\kappa(\bt_1)<p_1$. 
Hence there exists $D_1>0$ such that
$\ds u_n(\tb_1)\le D_1 e^{p_1n}$
and $\ds u_n(\bt_1)\le D_1 e^{p_1n}$. 
By the inequality 
$\ds  \log \xi_a \le \frac{\xi_a^{\bt_1}}{\bt_1}$ and \eqref{relation uv}, we get
\beqn
\sum_{n=1} e^{-pn} \sum_{a\in S_n} \rho_a\log \xi_a 
\le \frac{1}{\bt_1} \sum_{n=1}^\infty e^{-pn}  v_n(\bt_1) 
\le \frac{D_1 e^{H|\bt_1|}}{\bt_1} \sum_{n=1}^\infty e^{-(p-p_1)n}<\infty.
\eeqn
Similarly, using the inequality
$\ds \frac{\xi_a^{\tb_1}}{\tb_1}\le \log \xi_a$ and that $\tb<0$,
we get
 \beqn
\sum_{n=1} e^{-pn} \sum_{a\in S_n} \rho_a\log \xi_a 
\ge \frac{1}{\tb_1} \sum_{n=1}^\infty e^{-pn}  v_n(\tb_1) 
\ge \frac{D_1 e^{-H|\tb_1|}}{\tb_1} \sum_{n=1}^\infty e^{-(p-p_1)n}>-\infty.
\eeqn
Hence we have shown that $\ds -\infty<\int \opsi d\nu_{\varphi^+}<\infty$. 
The proof  is complete.
\end{proof}

We further set
\beqn
p_t:=P_L(\varphi + t\psi),
\eeqn
where $P_L(\cdot)$ is the variational liftable pressure given by \eqref{def lift pressure}, which is greater than
$-\infty$ but may be $+\infty$.
Since $\varphi$ is strongly regular, we have 
$p=p_0=P_L(\varphi)$ is a finite number.
Also, let $\lambda$ be the number given by \eqref{def lambda}. 
We further define a constant (with regard to $x$) by
\beq\label{def qt}
q_t = p+\lambda t,
\eeq
which can be viewed as a linear function in terms of $t$.

\begin{lemma}\label{lem: PSZ pt}
For any $t\in \IR$, we have
$\ds 
p_t\ge q_t.
$
\end{lemma}

\begin{proof} 
Since   
$\mu_\varphi$ is the unique equilibrium measure for  $\varphi$ in the class 
$\CM_L(f, Y)$ of liftable measures,  we have 
\beqn
p=P_L(\varphi) = 
h_{\mu_\varphi}(f)  + \int \varphi d\mu_\varphi. 
\eeqn
Hence 
\beqn
p_t=\sup_{\mu\in \CM_L(f, Y)} \left\{h_\mu(f) + \int (\varphi + t\psi) d\mu \right\}
\ge h_{\mu_\varphi}(f) + \int (\varphi + t\psi) d\mu_\varphi 
= q_t.
\eeqn
\end{proof}

Recall that the auxiliary function $\kappa: (-\infty, \infty)\to [-\infty, \infty]$ that we have  defined in \eqref{define kappa}. 
 
\begin{lemma}\label{lem: kappa1}
There are $\tb<0<\bt$ such that 
$\ds \kappa(t)<q_t$ 
for all $t\in (\tb, \bt)$.
\end{lemma}

\begin{proof} 

If $\kappa\equiv -\infty$, then we simply take $\tb=\tb_0$ and $\bt=\bt_0$.

Otherwise, by Lemma~\ref{lem: kappa},  
$\kappa$ is a real-valued convex continuous function on the interval $(\tb_0, \bt_0)$.
Take 
$\ds 
\kappa_1(t)=\kappa(t)-q_t.
$
Since $q_t$ is linear in $t$,
the function $\kappa_1$ is a 
continuous convex function on $(\tb_0, \bt_0)$, and
$\kappa_1(0)=\kappa(0)-p<0$.
Then this lemma holds if we set 
\beq\label{choose tb bt}
\tb:=\inf\{t> \tb_0: \ \kappa_1(t)<0\} \ \ \text{and} \ \ \bt:=\sup\{t< \bt_0: \ \kappa_1(t)<0\}.
\eeq
It is clear that $\tb<0<\bt$ by the continuity of $\kappa_1$.
\end{proof}

Now we are ready to prove Proposition \ref{lem: check P3-5}.

\begin{proof}[Proof of Proposition \ref{lem: check P3-5}]
Let $\tb<0<\bt$ be given by Lemma \ref{lem: kappa1}. 
Recall that $q_t$ is given by \eqref{def qt}. 
For any $t\in (\tb, \bt)$, by Lemma~\ref{lem: kappa1},  we can
choose $\eps_t$ such that { $0<2\eps_t<q_t-\kappa(t)$}.
By  Lemma~\ref{lem: kappa}, there is $D_t>0$ such that
\beq\label{est un}
u_n(t)\le D_t e^{n(\kappa(t)+\eps_t)}\le  D_t e^{n(q_t-\eps_t)}.
\eeq
We now verify Conditions \textbf{(P3)-(P5)} for the potential $\varphi+t\psi-q_t$.

\noindent (i) Note that
$\overline{\varphi + t\psi - q_t} = \overline{\varphi +t\psi} - q_t\tau$, 
by \eqref{est un}, we get
\beqn
\begin{split}
\sum_{a\in S} \sup_{x\in J_a} \exp (\overline{\varphi + t\psi- q_t}(x))
= \sum_{a\in S} \rho_a(t) e^{-q_t\tau(a)} 
& =\sum_{n=1}^\infty e^{-nq_t} u_n(t) \\
& \le D_t \sum_{n=1}^\infty e^{-n \eps_t}<\infty.
\end{split}
\eeqn
 Thus Condition \textbf{(P3)} holds.

\noindent
(ii) Since now Conditions \textbf{(P1)-(P3)} hold for 
the potential $\varphi+t\psi-q_t$, we have 
\beqn
p_t=P_L(\varphi+t\psi)=P_L(\varphi+t\psi-q_t) + q_t<\infty. 
\eeqn
Also note that 
\beqn
\begin{split}
(\varphi+t\psi-q_t)^+
& =\overline{\varphi+t\psi-q_t}-P_L(\varphi+t\psi-q_t)\tau \\
& =\overline{\varphi+t\psi} -q_t \tau-P_L(\varphi+t\psi)\tau +q_t \tau \\
&= \overline{\varphi+t\psi} - p_t \tau.
\end{split}
\eeqn
We take $\eps=\eps_t/2$, then by  \eqref{est un} and
Lemma~\ref{lem: PSZ pt}, 
Condition \textbf{(P4)} holds  since
\beqn
\begin{split}
\sum_{a\in S} \tau(a) \sup_{x\in J_a}\exp((\varphi+t\psi-q_t)^+(x)+\eps\tau(x)) 
= &  \sum_{a\in S} \tau(a) \rho_a(t) e^{\tau(a) (\eps-p_t) } \\
=&   \sum_{n=1}^\infty n e^{n(\eps-p_t)} u_n(t) \\
\le &   D_t \sum_{n=1}^\infty n e^{-n\eps_t/2}<\infty.
\end{split}
\eeqn

\noindent
(iii) Since now the potential $\varphi+t\psi-q_t$ satisfies Conditions \textbf{(P1)-(P4)}, 
by Theorem 4.5 in \cite{PS08},  
the normalized induced potential
$\ds (\varphi+t\psi-q_t)^+=\overline{\varphi+t\psi}-p_t\tau$
has a unique 
$F$-invariant equilibrium measure $\nu_t:=\nu_{(\varphi+t\psi-q_t)^+}$ in $\CM(F, W)$.
Moreover, $\nu_t$ has the Gibbs property, and in particular, 
there is $K_t>0$ such that for any $a\in J_a$, 
\beqn
\nu_t(J_a)\le K_t \sup_{x\in J_a} \exp (\varphi+t\psi-q_t)^+(x) \le 
K_t   \rho_a(t) e^{-p_t \tau(a)}
\eeqn
Hence by \eqref{est un} and
Lemma~\ref{lem: PSZ pt}, 
Condition \textbf{(P5)} holds since 
\beqn
\begin{split}
\nu_t(\{x\in W: \ \tau(x)\ge N\})
=\sum_{n\ge N} \sum_{a\in S_n} \nu_t(J_a)
&\le  K_t   \sum_{n\ge N} e^{-np_t} u_n(t) \\
&\le  K_tD_t   \sum_{n\ge N} e^{-n \eps_t} \\
& =C_t (e^{- \eps_t})^{N},
\end{split}
\eeqn
for some constant $C_t>0$. This completes the proof of Lemma \ref{lem: check P3-5}.
\end{proof}

Now we are ready to prove Theorem~A.

\begin{proof}[Proof of Theorem~A]
Proposition \ref{lem: check P1-2} and Proposition \ref{lem: check P3-5} 
show that Conditions \textbf{(P1)-(P5)} hold 
for the potential $\varphi+t\psi-q_t$ for all $t\in (\tb, \bt)$. 
In other words, $\varphi+t\psi-q_t$ is a strongly nice potential. 

It follows from Theorem \ref{thm: is0} that $\varphi+t\psi-q_t$ has a unique 
equilibrium measure $\mu_t:=\mu_{\varphi+t\psi-q_t}$ in the class $\CM_L(f, Y)$;
moreover, if $\{\cS, \tau\}$ satisfies the aperiodic condition~\eqref{def aperiodic}, 
then $\mu_{t}$ is mixing and in fact Bernoulli, combining with Condition \textbf{(P5)},
$\mu_{t}$ has exponential decay of correlations and satisfies the 
Central Limit Theorem. 
Therefore, Statement (1)(2) of Theorem~A  follows from the fact that 
$\varphi+t\psi$ and $\varphi+t\psi-q_t$ admit the same equilibrium measures.

We now show Statement (3) of Theorem~A. 
Since the potential $\varphi+t\psi-q_t$ satisfies Conditions \textbf{(P1)-(P3)}  for all $t\in (\tb, \bt)$,
then by Theorem 4.5 in \cite{PS08}, we have $P_L(\varphi+t\psi - q_t)<\infty$. Moreover, since 
\beqn
\varphi+(t+s)\psi - q_{t+s} = \left[(\varphi-p)+t(\psi-\lambda) \right] + s (\psi-\lambda) 
\eeqn
is nice for all $t\in (\tb, \bt)$ and sufficiently small $|s|$,
it follows from Theorem \ref{thm: SZ} that the function 
$
t\mapsto P_L (\varphi+t\psi - q_{t})
$
is real analytic in $(\tb, \bt)$. Therefore,  
the function
\beqn
t\mapsto P_L(\varphi+t\psi )=P_L(\varphi+t\psi - q_t) + p+\lambda t,
\eeqn
is also finite and real analytic in $(\tb, \bt)$. 
\end{proof}

\section{Liftability Problem}

\subsection{Compatible Partition and Upper bound for $P_\mu(\varphi)$}

In this subsection, we first explain Condition \textbf{(L1)} in Theorem~B.

\begin{definition}\label{def compatible partition}
We say that a measurable partition $\cP$ of $X$ is \emph{compatible} to the inducing scheme $\{\cS, \tau\}$ if 
the following property holds: 
for any $a\in S$ and any $0\le i<\tau(a)$, 
the set $f^i(J_a)$ is contained in an element of $\cP$.
\end{definition}

The compatibility condition implies that the iterates of any block $J_a$ would not be cut
by $\partial \cP$ into two or more pieces before it returns to the base $W$. 
Such condition is similar to Condition (C) in \cite{PSZ08}.

Let $\cP$ be the partition given by 
Condition \textbf{(L1)}, that is,  $\cP$ is a  generating partition 
which is compatible to the inducing scheme $\{\cS, \tau\}$. 
Note that $\cP$ is \emph{generating} means  
the smallest $\sigma$-algebra containing $\bigcup_{n\ge 0}\cP_n$ 
(or $\bigcup_{n\ge 0}  \wcP_n$ if $f$ is invertible)
is the Borel $\sigma$-algebra on $X$, where 
$\cP_n := \bigvee_{k=0}^n f^{-k} \cP$ (or $\wcP_n := \bigvee_{k=-n}^n f^{-k} \cP$).

Let $\varphi: X\to \IR$ be the potential given by Condition \textbf{(L2)}, that is,
the inducing scheme $\{\cS, \tau\}$ has finite $\ovarphi$-complexity,
and thus the complexity function $\cK(\ovarphi)$ is finite (see Definition~\ref{def finite psi comp}).

Theorem~B claims that under Conditions \textbf{(L1)(L2)},
all ergodic measures $\mu\in \CM(f, X)$ are liftable, i.e., $\mu\in \CM_L(f, Y)$, provided that
\begin{itemize}
\item $\mu$ gives positive weight to the base, i.e., $\mu(W)>0$;
\item $\mu$ has sufficiently large but finite pressure, i.e.,
$\cK(\ovarphi)<P_\mu(\varphi)<\infty$.
\end{itemize} 
Recall that $P_\mu(\varphi)$ is the pressure of $\varphi$ with respect to $\mu$ (see \eqref{def Pmu}).
As we assume that $f$ has finite variational entropy (see \eqref{def entropy}),
the condition that $-\infty<P_\mu(\varphi)<\infty$ is equivalent to that $\ds  \int |\varphi| d\mu<\infty$.
Hence
it suffices to consider ergodic measures  $\mu\in \CM(f, X)$ such that $\varphi$ is $\mu$-integrable.

To effectively estimate the pressure $P_\mu(\varphi)$ from above, 
we introduce the following Caratheodory-Pesin type quantity.   
Given an integer $m\ge 1$, 
let $\cG_m$ be the collection of all elements in $\cup_{n\ge m} \cP_n$
whose depth are marked, that is,
\beqn
\cG_m:=\left\{(A, n): \ n\ge m \ \text{and} \ A\in\cP_n \right\}.
\eeqn
Given a Borel subset $Z\subset X$ and
a real number $\alpha\in \IR$, we let 
\beq\label{def M}
M(Z, \varphi, \alpha, m):=
\inf_{\cG} \left\{ \sum_{(A, n)\in \cG} 
\exp\left(-\alpha n + \sup_{x\in A\cap Z} S_n\varphi(x)  \right)\right\},
\eeq
where 
$S_n\varphi:=\sum_{k=0}^{n-1} \varphi\circ f^k$ and 
the infimum is taken over all sub-collection $\cG$ of $\cG_m$ which covers $Z$, that is,
$\ds Z\subset \bigcup_{(A, n)\in \cG} A$. 
Since $M(Z, \varphi, \alpha, m)$ is non-decreasing in terms of $m$, 
we then define
\beqn
M(Z, \varphi, \alpha):=\lim_{m\to \infty} M(Z, \varphi, \alpha, m).
\eeqn 
Moreover, we define the \emph{pressure of $\varphi$ on $Z$} by
\beqn
P_Z(\varphi):=\inf\left\{\alpha\in \IR: \ M(Z, \varphi, \alpha)=0\right\}
=\sup\left\{\alpha\in \IR: \  M(Z, \varphi, \alpha)>0\right\}.
\eeqn

We remark that
our definition of $P_Z(\varphi)$ is slightly different from the standard definition given in Section 11 of \cite{Pe97},
as the collection $\cG$ is not taken as an open cover but a sub-collection related to the generating partition $\cP$.  
Moreover, 
since the map $f$ and the potential  $\varphi$
that we consider are not assumed to be continuous, 
we may not have the \emph{inverse variational principle}, i.e., 
$\ds P_\mu(\varphi)=\inf\{P_Z(\varphi): \mu(Z)>0\}$.
Nevertheless, 
the following lemma is enough for our purpose, 
which provides an upper bound for $P_\mu(\varphi)$. 
 
\medskip 
 
\begin{lemma}\label{lem: est Q}
For any ergodic measure $\mu\in \CM(f, X)$ with $\ds \int |\varphi| d\mu<\infty$, 
and for any Borel subset $Z\subset X$ with $\mu(Z)>0$, 
we have
\beqn
P_\mu(\varphi) \le P_Z(\varphi). 
\eeqn
\end{lemma}

\begin{proof}
For any $\eps>0$, we set $\ds \alpha_\eps=P_\mu(\varphi)-2\eps$. 
Recall that $\cP$ is a generating partition, and hence $h_\mu(f)=h_\mu(f, \cP)$.

Since $\mu$ is ergodic, by the Birkhoff's ergodic theorem and the
Shannon-McMillan-Breiman theorem,
there exists $m_0\in \IN$ and 
a Borel subset $X_0\subset X$ with $\mu(X\backslash X_0)<\frac12\mu(Z)$,
such that for all $n\ge m_0$ and $x\in X_0$, 
\beqn
S_n\varphi(x)\ge n\left(\int \varphi d\mu -\eps \right)
\ \ \text{and}  \ \ 
\log \mu\left(\cP_n(x) \right) \le -n\left( h_\mu(f) - \eps \right),
\eeqn
where $\cP_n(x)$ denotes the element of $\cP_n$ containing $x$.  

We take $Z_0=Z\cap X_0$, then $\mu(Z_0)>0$. Given any $m\ge m_0$,
let $\cG$ be a sub-collection of $\cG_m$ which covers $Z_0$.
For any $(A, n)\in \cG$ with $A\cap Z_0\ne \emptyset$, 
we have $A=\cP_n(x)$ for any $x\in A\cap Z_0$, and hence 
\beqn
\begin{split}
\sum_{(A, n)\in \cG} \exp\left( -\alpha_\eps n + \sup_{x\in A\cap Z} S_n\varphi(x)  \right)
&\ge \sum_{(A, n)\in \cG}\exp\left( -\alpha_\eps n + \sup_{x\in A\cap Z_0} S_n\varphi(x)  \right) \\
&\ge \sum_{(A, n)\in \cG}\exp\left( n\left( \int \varphi d\mu -\eps -\alpha_\eps\right) \right) \\
& = \sum_{(A, n)\in \cG}\exp \left(- n\left( h_\mu(f) -\eps \right) \right) \\
&\ge \sum_{(A, n)\in \cG}  \mu(A)\ge \mu(Z_0).
\end{split}
\eeqn
Therefore, 
\beqn
M(Z, \varphi, \alpha_\eps)=\lim_{m\to \infty} M(Z, \varphi, \alpha_\eps, m)\ge
\lim_{m\to \infty} M(Z_0, \varphi, \alpha_\eps, m)\ge \mu(Z_0)>0,
\eeqn
which implies that $P_Z(\varphi)\ge \alpha_\eps=P_\mu(\varphi)-2\eps$.
Thus $P_Z(\varphi)\ge  P_\mu(\varphi)$ since $\eps$ is arbitrarily chosen. 
\end{proof}

\subsection{Proof of Theorem~B}

Section 5 in \cite{PSZ16} provides a necessary condition for an ergodic measure 
to be non-liftable. We briefly describe it below.

Define a tower associated with the inducing scheme $\{\cS, \tau\}$ by
\beqn
\hY:=\left\{(x, k): x\in W, \ 0\le k<\tau(x) \right\}\subset W\times \IN
\eeqn
and define $\hf: \hY\to \hY$ by 
\beqn
\hf(x, k):=
\begin{cases}
(x, k+1),  & \ k<\tau(x)-1, \\
(F(x), 0),  & \ k=\tau(x)-1. 
\end{cases}
\eeqn
For any $N\ge 1$, we set 
\beqn
\hE_N:=\left\{(x, k)\in \hY:  \ 0\le k\le N \right\}.
\eeqn
For any $x\in W$ and any $n\in \IN$, we define the 
frequency of the $\hf$-orbit of the point $(x, 0)$ falling into $\hE_N$ 
during the first $n$ iterates by
\beqn 
A_n^N(x):=\frac1n\#\left\{0\le j<n:\ \hf^j((x, 0))\in \hE_N\right\}.
\eeqn
Following the ideas from Keller \cite{Ke89} and
Zweim\"uller \cite{Zwe05},
Pesin, Senti and Zhang obtained the following property for non-liftable measures (Lemma 5.4 in \cite{PSZ16}). 

\begin{lemma}\label{lem: freq} 
Let $\mu\in \CM(f, X)$ be such that $\mu(W)>0$.
If  $\mu$ is non-liftable,
then there exists an increasing sequence $\{n_k\}_{k\ge 1}$ of positive integers such that for any $\eps>0$, 
there exists $Z\subset W$ satisfying that $\mu(Z)\ge (1-\eps)\mu(W)$ and
\beqn
\lim_{k\to\infty} A_{n_k}^N(x)=0 \ \ \text{uniformly on}\ Z \ \text{for all}\ N\ge 1. 
\eeqn
\end{lemma}
 
In fact, \cite{PSZ16} considered an frequency $A_n^N$ falling into a larger set containing $\hE_N$, and thus 
Lemma~\ref{lem: freq} here is a corollary of  Lemma 5.4 in \cite{PSZ16}.
 
Later in the proof of Theorem~B, 
we shall use the following combinatoric inequality (see e.g. Section I.5 of \cite{Shields96}):  
for any $n\ge m\ge 0$, we have 
\beq\label{combinatoric}
 {n \choose    m}\le \exp\left(nh\left(\frac{m}{n}\right)\right)
< \exp(n),
\eeq
where 
\beq\label{def h}
h(t):=-t\log t-(1-t)\log(1-t), \ \text{for any} \ t\in [0, 1].
\eeq
Note that  $h(t)$ is increasing for $t\in [0, \frac12]$
and $\ds \max_{t\in [0, 1]} h(t)=\log 2<1$.

We are now ready to prove Theorem~B.

\begin{proof}[Proof of Theorem~B]
Let $\mu\in \CM(f, X)$ be an ergodic   measure
with $\mu(W)>0$ and $\cK(\ovarphi)<P_\mu(\varphi)<\infty$. 
Fix some $\cK\in \left(\cK(\ovarphi),  P_\mu(\varphi)\right)$.
By Definition~\ref{def finite psi comp}, 
we can pick $D>1$ and $n_0\in \IN$ such that 
\beq\label{est Un}
\begin{split}
U_n(\ovarphi) \le De^{\cK n} \   \text{for all}\ n\ge 1,  \\
U_n(\ovarphi) \le e^{\cK n} \  \  \text{for all}\ n\ge n_0.
\end{split}
\eeq
We fix a sufficiently small $\delta\in (0, \frac{1}{10})$ such that
\beq\label{choose delta}
\cK_1:= \cK +  3\delta+h(2\delta) +2\delta\log D  <  P_\mu(\varphi),
\eeq
 where $h(\cdot)$ is the function given by \eqref{def h}.

Now suppose that $\mu$ is non-liftable. 
We shall prove Theorem~B by contradiction.

Fix
$\ds
N=\lceil 1/\delta\rceil.
$
By Lemma~\ref{lem: freq}, 
there exists
a Borel subset $Z\subset W$ with 
$\mu(Z)>0$ 
and an increasing sequence $\{n_k\}_{k\ge 1}$ of positive integers  
(by choosing a subsequence of the original sequence in Lemma~\ref{lem: freq} if necessary)
such that 
\beq\label{A delta}
\sup_{x\in Z} A_{n_k}^N(x)< \delta/2, \ \ \text{for all} \ k\ge 1.
\eeq
Without loss of generality, we may assume that for all $k\ge 1$, 
\beq\label{def nk}
\begin{split}
(i)\ \ & n_k>\max\{n_0,  100N/\delta\}, \ \text{where}\ n_0\ \text{is  given by \eqref{est Un}}; \\
(ii)\ & n^3\le e^{\delta n} \ \text{for all}\ n\ge n_k. 
\end{split}
\eeq

We would like
to estimate the quantity $M\left(Z, \varphi, \alpha, n_k\right)$ introduced in \eqref{def M}.
To this end,
we construct a particular covering of $Z$ as follows.
  
An \emph{$n_k$-code} via the inducing scheme $\{\cS, \tau\}$ is the $s$-tuple of the form
\beqn
\ba=(a_0, \dots, a_{s-1}) \in S^{s}
\eeqn
such that $m_{s-1}< n_k\le m_s$, 
where $m_0=0$ and 
$
m_j=\sum_{i=0}^{j-1} \tau(a_i) 
$
for $1\le j\le s$.
For convenience, we denote
$\ds n(\ba)=m_s$
and 
$\ds \btau(\ba)=(\tau(a_0), \dots, \tau(a_{s-1}))$.

Condition \textbf{(L1)} assumes that the partition $\cP$ is compatible to the inducing scheme,
which  implies that for any $a\in S$, the block $J_a$ is contained in an element of $\cP_{\tau(a)}$.
Hence an $n_k$-code $\ba$ defines a Borel subset
\beqn
A_\ba:=J_{a_0} \cap f^{-m_1} J_{a_1}\cap \dots \cap f^{-m_{s-1}}J_{a_{s-1}},
\eeqn
which is contained in an element of $\cP_{n(\ba)}$. 

Subject to the Borel subset $Z$ and the sequence $\{n_k\}_{k\ge 1}$ that are
obtained from Lemma~\ref{lem: freq},  we make the following notions:
\begin{itemize}
\item
let $\Gamma(Z, n_k)$ be the collection of all $n_k$-codes $\ba$ such that
$A_{\ba}\cap Z\ne \emptyset$. 
\item
we say that a string $\btau=(\tau_0, \dots, \tau_{s-1})$ is \emph{$(Z, n_k)$-admissible} 
if there is $\ba\in \Gamma(Z, n_k)$ such that $\btau(\ba)=\btau$. 
Let $\cA(Z, n_k)$ be the collection of $(Z, n_k)$-admissible strings.
\item
given a $(Z, n_k)$-admissible string $\btau$, we 
let $\Gamma(Z, \btau)$ be the collection of all $n_k$-codes $\ba\in \Gamma(Z, n_k)$ such that $\btau(\ba)=\btau$.
It is clear that 
$
\Gamma(Z, n_k)=\bigcup_{\btau\in \cA(Z, n_k)} \Gamma(Z, \btau).
$
\item given a $(Z, n_k)$-admissible string $\btau=(\tau_0, \dots, \tau_{s-1})$, 
we denote $n(\btau)=\sum_{i=0}^{s-1}\tau_i$.
We claim that 
\beq\label{tau quotient}
\dfrac{\sum_{\tau_i\le N}\ \tau_i}{n(\btau)}<\delta, 
\ \text{or equivalently}, \
\dfrac{\sum_{\tau_i> N}   \ \tau_i}{n(\btau)}\ge 1-\delta.
\eeq
Indeed, let $n':=\sum_{0\le i<s-1: \ \tau_i\le N} \tau_i$, then \eqref{A delta} implies that 
$\ds
n'/n_k<\delta/2,
$
and hence
\beqn
\dfrac{\sum_{\tau_i\le N}\ \tau_i}{n(\btau)}\le \dfrac{n'+N}{n(\btau)}\le \dfrac{n'+N}{n_k}
\le \delta/2+\delta/100<\delta.
\eeqn
\end{itemize}

It is clear that $Z\subset \bigcup_{\ba\in \Gamma(Z, n_k)} A_{\ba}
=\bigcup_{\btau\in \cA(Z, n_k)} \bigcup_{\ba\in \Gamma(Z, \btau)} A_{\ba}$,
where each $A_{\ba}$ lies inside an element of $\cP_{n(\ba)}$. 
Hence for any $\alpha\in \IR$,  we have
\beqn
\begin{split}
M(Z, \varphi, \alpha, n_k) 
&\le \sum_{\ba\in \Gamma(Z, n_k)} \exp\left(-\alpha n(\ba) + \sup_{x\in A_{\ba}\cap Z} S_{n(\ba)}\varphi(x)  \right) \\
& = \sum_{\btau\in \cA(Z, n_k)} e^{-\alpha n(\btau)} 
\sum_{\ba\in \Gamma(Z, \btau)} \exp\left(  \sup_{x\in A_{\ba}\cap Z} S_{n(\btau)}\varphi(x)  \right).
\end{split}
\eeqn
Note that for any  $\btau=(\tau_0, \dots, \tau_{s-1})\in \cA(Z, n_k)$
and any 
$\ba=(a_0, \dots, a_{s-1})\in \Gamma(Z, \btau)$,  
we set $m_0=0$ and  $m_j=\sum_{i=0}^{j-1} \tau_i$ for $1\le j\le s$, then
we have
\beqn
\sup_{x\in A_{\ba}\cap Z} S_{n(\btau)}\varphi(x)
= \sup_{x\in A_{\ba}\cap Z}
\sum_{i=0}^{s-1} \ovarphi( f^{m_i}(x))   \le \sum_{i=0}^{s-1} \sup_{x\in J_{a_i}} \ovarphi(x).
\eeqn
Recall that $U_n(\cdot)$ given by Definition~\ref{def finite psi comp}, then by \eqref{est Un} and 
\eqref{tau quotient}, we have
\beqn
\begin{split}
\sum_{\ba\in \Gamma(Z, \btau)} \exp\left(  \sup_{x\in A_{\ba}\cap Z} S_{n(\btau)}\varphi(x)  \right)
& \le \sum_{\ba\in \Gamma(Z, \btau)} \prod_{i=0}^{s-1} \exp\left( \sup_{x\in J_{a_i}} \ovarphi(x) \right) \\
& \le \prod_{i=0}^{s-1} U_{\tau_i}(\ovarphi)  \\
&\le \prod_{\tau_i\le N} De^{\cK\tau_i} \prod_{\tau_i> N} e^{\cK\tau_i} \\
&\le \exp\left( n(\btau) \left( \cK + \delta\log D\right) \right).
\end{split}
\eeqn
Therefore, 
\beqn
\begin{split}
M(Z, \varphi, \alpha, n_k) 
&\le \sum_{\btau\in \cA(Z, n_k)} 
\exp\left( n(\btau) \left( \cK-\alpha + \delta\log D\right) \right)\\
&=\sum_{n=n_k}^\infty \gamma_n \exp\left( n\left( \cK-\alpha + \delta\log D\right) \right),
\end{split}
\eeqn
where $\gamma_n:=\#\{\btau\in \cA(Z, n_k): \ n(\btau)=n\}$.
Note that $\gamma_n$ is no more than
the number of ways to rewrite $n$ as
\beqn
n=\tau_0+\tau_1+\dots+\tau_{s-1}
\eeqn 
such that each $\tau_i\ge 1$ and $\sum_{\tau_i\le N} \ \tau_i<\delta n$. 
Let $s_1$ be the number of $i$'s such that $\tau_i\le N$, 
and then $s_2=s-s_1$ is the number of $i$'s such that $\tau_i> N$.
Then we have 
$\ds
0\le s_1< \delta n
$
and 
$\ds
1\le s_2\le n/N<\delta n
$.
Therefore, by the combinatoric inequality \eqref{combinatoric} and 
the monotonicity of the function $h(t)$ given by \eqref{def h} for $t\in [0, \frac12]$, 
as well as the choices of $\delta$ and $n_k$ given by \eqref{choose delta} and \eqref{def nk} respectively,
we get for any $n\ge n_k$, 
\beqn
\begin{split}
\gamma_n
& \le \sum_{s_1=0}^{\lfloor \delta n\rfloor} \sum_{s_2=1}^{\lfloor \delta n\rfloor}
{n-s_2 N-1 \choose s_1+s_2-1} { s_1+s_2 \choose s_1} \\
&\le \sum_{s_1=0}^{\lfloor \delta n\rfloor} \sum_{s_2=1}^{\lfloor \delta n\rfloor}
{n  \choose s_1+s_2-1} { s_1+s_2 \choose s_1}\\
&\le \sum_{s_1=0}^{\lfloor \delta n\rfloor} \sum_{s_2=1}^{\lfloor \delta n\rfloor}
\exp\left( n h\left(\frac{s_1+s_2-1}{n}\right)\right) 
\exp\left( s_1+s_2  \right) \\
&\le \sum_{s_1=0}^{\lfloor \delta n\rfloor} \sum_{s_2=1}^{\lfloor \delta n\rfloor}
\exp\left( n h\left(2\delta\right)\right) 
\exp\left( 2\delta n  \right) \\
&\le 2(\delta n)^2 \cdot  \exp\left(n \left(2\delta + h(2\delta)\right) \right) 
< \exp\left(n(3\delta + h(2\delta))\right).
\end{split}
\eeqn
Finally, let $\cK_1$ given by \eqref{choose delta} and take some $\alpha\in (\cK_1, P_\mu(\varphi))$, then
we obtain
\beqn
\begin{split}
M(Z, \varphi, \alpha, n_k) 
& \le  \sum_{n=n_k}^\infty  \exp\left( n\left( \cK-\alpha + 3\delta+ h(2\delta) +\delta\log D \right) \right) \\
&\le \sum_{n=n_k}^\infty \exp\left(n\left(\cK_1 - \alpha\right)\right)
\le \dfrac{\exp\left(n_k(\cK_1-\alpha) \right)}{1-\exp(\cK_1-\alpha)},
\end{split}
\eeqn
which implies that $M(Z, \varphi, \alpha)=\lim_{k\to\infty} M(Z, \varphi, \alpha, n_k) =0$
and thus $P_Z(\varphi)\le \alpha<P_\mu(\varphi)$. 
However, by Lemma~\ref{lem: est Q}, we must have $P_\mu(\varphi)\le P_Z(\varphi)$,
which is a contradiction. 
Therefore, the measure $\mu$ has to be liftable. 
The proof of Theorem~B is complete.
\end{proof}

\subsection{Proof of Corollary~C}

We now proceed the proof of Corollary~C.

\begin{proof}[Proof of Corollary~C] 
Let $\{\cS, \tau\}$ be an inducing scheme with inducing domain $W$.
Let $\varphi$ be a strongly nice potential satisfying Conditions \textbf{(P1)-(P5)}
and let $\psi$ be a regular potential satisfying Conditions \textbf{(P1)(P2)}.
Recall that $\CM_L(f, Y)$, 
$\CM^W(f, X)$, $P_L(\cdot)$ and $P^W(\cdot)$ are defined by 
\eqref{def lift class}
\eqref{def CM*}, \eqref{def lift pressure} and \eqref{def v pressure} respectively. 

If Condition \textbf{(H)} holds, then by Statement (1) of Theorem~A, we
let $\mu_t$ be the unique equilibrium measure for the potential $\varphi+t\psi$
in the class $\CM_L(f, Y)$, then 
$\ds P_{\mu_t}(\varphi+t\psi)=P_L(\varphi+t\psi)$
is finite for all $t\in (\tb, \bt)$. 
It follows that $\ds \int |\varphi+t\psi| d\mu_t<\infty$, and hence 
$P_L(\varphi+t\psi)\le P^W(\varphi+t\psi)$ for all $t\in (\tb, \bt)$.

If now Condition \textbf{(L1)} holds,
Lemma~\ref{lem: kappa}, Lemma~\ref{lem: PSZ pt} and Lemma~\ref{lem: kappa1} imply that 
the inducing scheme $\{\cS, \tau\}$ has finite 
$\overline{\varphi+t\psi}$-complexity, and 
\beqn
-\infty<\cK\left( \overline{\varphi+t\psi}\right) =\kappa(t)<q_t\le p_t=P_L(\varphi+t\psi)<\infty.
\eeqn
For any ergodic measure $\mu_*\in \CM^W(f, X)$ such that
\beqn
\begin{split}
P_{\mu_*}(\varphi+t\psi) 
& = P^W(\varphi+t\psi) \\
& =\sup\left\{P_\mu(\varphi+t\psi): \ \mu\in \CM^W(f, X) \ \text{and} \ \int |\varphi+t\psi| d\mu <\infty \right\},
\end{split}
\eeqn
we have 
\beqn
\cK\left( \overline{\varphi+t\psi}\right)<P_L(\varphi+t\psi)\le 
P_{\mu_*}(\varphi+t\psi)\le h(f)+\int |\varphi+t\psi| d\mu_*<\infty.
\eeqn
Then by Theorem~B, the measure $\mu_*$ is liftable and hence $\mu_*=\mu_t$. 
In other words,  $\mu_t$  is in fact 
the unique equilibrium measure for the potential $\varphi+t\psi$ 
in the class $\CM^W(f, X)$.
Moreover, by Statement (3) of Theorem~A,
the function
$
\ds t\mapsto P^W(\varphi+t\psi)=P_L(\varphi+t\psi)
$
is finite and real analytic in $(\tb, \bt)$.
\end{proof}

\medskip

\section*{Acknowledgements}

J. Chen is partially supported by the National Key Research and Development Program of China (No.2022YFA1005802),
the NSFC Grant 12001392 and NSF of Jiangsu BK20200850.
F.~Wang is supported by NSFC grant 11871045 and
the State Scholarship Fund from China Scholarship Council (CSC).
H.-K.~Zhang is partially supported by the 
Simons Foundation Collaboration Grants for Mathematicians (706383).


\begin{thebibliography}{BSC90}

\bibitem{Bow08} R.~Bowen.
\emph{Equilibrium states and the ergodic theory of Anosov diffeomorphisms},
{ Lecture Notes in Mathematics}, \textbf{470}. Springer-Verlag, Berlin, 2008.














\bibitem{Bruin95}
H.~Bruin. 
\emph{Induced maps, Markov extensions and invariant measures in one-dimensional dynamics.}
{Comm. Math. Phys.} \textbf{168}(1995), no. 3, 571--580.

\bibitem{BrTo07}
H.~Bruin and M.~Todd. 
\emph{Markov extensions and lifting measures for complex polynomials.}
{Ergodic Theory Dynam. Systems.} \textbf{27}(2007), no. 3,  743--768.

 

\bibitem{BS80}
L.~A.~Bunimovich and Ya.~G.~Sinai
\emph{Markov partitions for dispersing billiards},
{Commun. Math. Phys.} \textbf{73} (1980), 247--280.


\bibitem{BSC90}
L.~A.~Bunimovich, Ya.~G.~Sinai, and N.~I.~Chernov.
\emph{Markov partitions for two-dimensional hyperbolic billiards},
{Russian Math. Surveys} \textbf{45} (1990), 105--152.

\bibitem{Buzzi99}
J.~Buzzi. 
\emph{Markov extensions for multi-dimensional dynamical systems.}
{Israel J. Math.} \textbf{112}(1999), 357--380.





























\bibitem{CZ05a}
N.~I.~Chernov and H.-K.~Zhang.
\emph{Billiards with polynomial mixing rates},
{Nonlineartity} \textbf{4} (2005), 1527--1553.


\bibitem{CZ05b} N.~Chernov and H.-K.~Zhang.
\emph{A family of chaotic billiards with variable mixing rates},
{Stochast. Dynam.} \textbf{5} (2005), 535--553.








\bibitem{De10}  M. Demers.
Functional norms for Young towers,
\emph{Ergod. Th. Dynam. Syst.} \textbf{30} (2010), no. 5, 1371--1398.





\bibitem{DZ13} M.~Demers and H.-K.~Zhang.
\emph{A functional analytic approach to perturbations of the Lorentz gas},
{Commun. Math. Phys.}  \textbf{324} (2013), 767--830.


\bibitem{DZ14} M. Demers and H.-K. Zhang.
\emph{Spectral analysis of hyperbolic systems with singularities},
{Nonlinearity} \textbf{27}   (2014), 379--433.




























\bibitem{Hof79}
F.~Hofbauer. 
\emph{On intrinsic ergodicity of piecewise monotonic transformations with positive entropy.}
{Israel J. Math.} \textbf{34}(1979), 213--237.

\bibitem{Hof81}
F.~Hofbauer. 
\emph{On intrinsic ergodicity of piecewise monotonic transformations with positive entropy. II.}
{Israel J. Math.} \textbf{38}(1981), 107--115.





\bibitem{Ke89}
G.~Keller. 
\emph{Lifting measures to Markov extensions.}
{Monatsh. Math.} \textbf{108}(1989), 183--200.







 






\bibitem{Pe97} Ya.~Pesin.
\emph{Dimension theory in dynamical systems},
{Contemporary views and applications. Chicago Lectures in Mathematics. 
University of Chicago Press, Chicago, IL, 1997. }


\bibitem{PS05} Ya.~Pesin and S.~Senti.
\emph{Thermodynamic formalism assocaited with inducing schemes for one-dimensional maps},
{Mosc. Math. J.} \textbf{5} (2005), no. 3, 669-678, 743-744.


\bibitem{PS08} Ya.~Pesin and S.~Senti.
\emph{Equilibrium measures for maps with inducing schemes},
{J. Mod. Dyn.} \textbf{2} (2008), no. 3, 397--430.

\bibitem{PSZ08} Ya.~Pesin, S.~Senti, and K.~Zhang.
\emph{Lifting measures to inducing schemes},
{Ergod. Th. Dynam. Syst.} \textbf{28} (2008), no. 2, 553--574.

\bibitem{PSZ16} Ya.~Pesin, S.~Senti, and K.~Zhang.
 \emph{Thermodynamics of towers of hyperbolic type},
{Trans. Amer. Math. Soc.}  \textbf{368} (2016), no. 12,  8519--8552.




\bibitem{Ru78} D.~Ruelle.
\emph{Thermodynamic formalism},
{Encyclopedia of Mathematics and its Applications},
vol 5,  Addison-Wesley, 1978.





\bibitem{Sin72} Ya.~G.~Sinai.
\emph{Gibbs measures in ergodic theory},
{Russ. Math. Surv.} \textbf{27} (1972), 21--69.





\bibitem{Sarig99} O.~Sarig.
\emph{Thermodynamic formalism for countable Markov shifts},
{Ergod. Th. Dynam. Syst.} \textbf{19} (1999), no. 6, 1565--1593.

\bibitem{Sarig01} O.~Sarig.
\emph{Thermodynamic formalism for null recurrent potentials},
{Israel J. Math.} \textbf{121} (2001), 285--311.

 

\bibitem{Sarig03} O.~Sarig.
\emph{Characterization of existence of Gibbs measures for countable Markov shifts},
{Proc. Amer. Math. Soc.} \textbf{131:6} (2003), 1751--1758.



\bibitem{ShZe19} F.~Shahidi and  A.~Zelerowicz.
\emph{Thermodynamics via inducing}, 
{J. Stat. Phys.} \textbf{175} (2019), no. 2, 351--383.





\bibitem{Shields96} P.~Shields.
\emph{The ergodic theory of discrete sample paths},
{Graduate Studies in Mathematics}, \textbf{13}. American Mathematical Society, Providence, RI, 1996. 










\bibitem{Y98} L.~S.~Young.
\emph{Statistical properties of dynamical systems with some hyperbolicity},
{ Ann. Math.} \textbf{147} (1998), 585--650.

\bibitem{Y99}  L.~S.~Young.
\emph{Recurrence times and rates of mixing},
{Israel J. Math.} \textbf{110} (1999), 153--188.



\bibitem{Zwe05} R.~Zweim\"uller.
\emph{Invariant measures for general(ized) induced transformations.}
{Proc. Amer. Math. Soc.} \textbf{133} (2005), no. 8, 2283--2295.



\end{thebibliography}
\end{document}